\documentclass[preprint,12pt]{elsarticle}

\usepackage{lineno,hyperref}
\modulolinenumbers[5]

\usepackage{ifpdf}
\usepackage{bm}
\usepackage{mathdots}
\usepackage{soul, color}
\usepackage{amsmath}
\usepackage{mathrsfs}
\usepackage{subfig}
\usepackage{graphicx}
\usepackage{epstopdf}
\usepackage{epsfig}
\usepackage{yfonts}

\newtheorem{theorem}{Theorem}[section]
\newtheorem{lemma}[theorem]{Lemma}

\newtheorem{definition}{Definition}[section]

\newenvironment{proof}[1][Proof]{\begin{trivlist}
\item[\hskip \labelsep {\bfseries #1}]}{\end{trivlist}}

\newenvironment{remark}[1][Remark]{\begin{trivlist}
\item[\hskip \labelsep {\bfseries #1}]}{\end{trivlist}}

\usepackage{amssymb}
\begin{document}

\begin{frontmatter}



\title{A note on the spectral distribution of symmetrized Toeplitz sequences }

\author{Sean Hon\corref{cor1}}
\address{Mathematical Institute, University of Oxford, Radcliffe Observatory Quarter, Oxford, OX2 6GG,
United Kingdom}
\ead{hon@maths.ox.ac.uk}
\cortext[cor1]{Corresponding author}

\author{Mohammad Ayman Mursaleen\corref{cor2}}
\address{Department of Science and high Technology, University of Insubria, Via Valleggio 11, Como, 22100,
Italy}
\ead{mamursaleen@uninsubria.it}

\author{Stefano Serra-Capizzano\corref{cor2}}
\address{Department of Science and high Technology, University of Insubria, Via Valleggio 11, Como, 22100,
Italy}
\ead{stefano.serrac@uninsubria.it}

\begin{abstract}

The singular value and spectral distribution of Toeplitz matrix sequences with Lebesgue integrable generating functions is well studied. Early results were provided in the classical Szeg{\H{o}} theorem and the Avram-Parter theorem, in which the singular value symbol coincides with the generating function. More general versions of the theorem were later proved by Zamarashkin and Tyrtyshnikov, and Tilli. Considering (real) nonsymmetric Toeplitz matrix sequences, we first symmetrize them via a simple permutation matrix and then we show that the singular value and spectral distribution of the symmetrized matrix sequence can be obtained analytically, by using the notion of approximating class of sequences. In particular, under the assumption that the symbol is sparsely vanishing, we show that roughly half of the eigenvalues of the symmetrized Toeplitz matrix (i.e. a Hankel matrix) are negative/positive for sufficiently large dimension, i.e. the matrix sequence is symmetric (asymptotically) indefinite.

\end{abstract}

\begin{keyword}
Toeplitz matrices \sep Hankel matrices \sep circulant preconditioners %\sep PCG \sep PMINRES

\MSC[] 15B05 \sep 65F15 \sep 65F08
\end{keyword}

\end{frontmatter}

%\linenumbers


\section{Introduction}
The singular value and spectral distribution of Toeplitz matrix sequences has been of interest over the past few decades. The earliest result on the eigenvalue distribution of Toeplitz matrices was established by Szeg{\H{o}} in \cite{MR890515}, namely the eigenvalues of the Toeplitz matrix $T_n[f]$ generated by a real-valued $f\in L^{\infty}([-\pi,\pi])$ are asymptotically distributed as $f$. Considering the same class of functions, Avram and Parter \cite{MR952991,MR851935} showed that the singular values of $T_n[f]$ are distributed as $|f|$. Tyrtyshnikov \cite{Tyrtyshnikov19961,MR1258226} later generalized the result for $T_n[f]$ generated by $f\in L^2([-\pi,\pi])$. Zamarashkin and Tyrtyshnikov \cite{MR1481397}, and Tilli \cite{MR1671591} further weakened the requirement on $f$ and showed that the same result holds for $f\in L^1([-\pi,\pi])$. Based on an approximation class sequence approach, Garoni, Serra-Capizzano, and Vassalos \cite{MR3399336} recently provided the same theorem for $f\in L^1([-\pi,\pi])$ in the framework of the newly developed theory of Generalized Locally Toeplitz (GLT) sequences \cite{MR3674485}. Moreover, Tyrtyshnikov in \cite{TYRTYSHNIKOV1994225} studied the corresponding change in the singular value and spectral distribution of Toeplitz matrix sequences when certain matrix operations are applied. In this direction, the algebra of matrix sequences generated by Toeplitz sequences is studied in \cite{CAPIZZANO2001121,glt-1,glt-2}.

Instead of directly dealing with a (real) nonsymmetric Toeplitz matrix $T_n\in \mathbb{R}^{n \times n}$, Pestana and Wathen \cite{doi:10.1137/140974213} recently suggested that one can premultiply $T_{n}$ by the anti-identity $Y_{n}$ defined as
\[  Y_{n}=\begin{bmatrix}{}
 & & 1 \\
 & \iddots & \\
1 &  & \end{bmatrix} \in \mathbb{R}^{n \times n}
\]
in order to obtain the symmetrized matrix $Y_{n}T_n$ (i.e. a Hankel matrix). Using a suitable absolute value circulant matrix $|C_n|$ as a preconditioner, they further proved that the eigenvalues of $|C_n|^{-1}Y_nT_n$ are clustered at $\pm1$, provided that certain conditions are satisfied. The same techniques were later shown to be applicable to certain Toeplitz-related matrices in \cite{Hon2018148,Hon2018}.

Considering the symmetrized Toeplitz matrix sequences $\{Y_nT_n[f]\}_n$, where $T_n[f]$ is generated by $f \in L^1[-\pi,\pi]$, we show that its singular value distribution can be obtained analytically. Moreover, while its eigenvalues are of course real, their modulus coincides with the singular value and precise information on the distribution of their sign can be provided. Specifically, we show that roughly half of the eigenvalues of $Y_nT_n[f]$ are negative/positive, when the dimension is sufficiently large and $f$ is sparsely vanishing,  i.e. its set of zeros is of (Lebesgue) measure zero. Our approach is based on the approximation class sequences introduced in the theory of GLT sequences (see the original definition in \cite{CAPIZZANO2001121} and many examples of its use in \cite{MR3674485}).

We finally stress that the spectral analysis of $\{Y_nT_n[f]\}_n$ is performed by using a new trick, whose generality goes far beyond the specific case under consideration, and such a tool can open a door to further research such as the classical eigenvalue distribution analysis.

\section{Preliminaries on Toeplitz matrices}
We assume that the given $n \times n$ Toeplitz matrix $T_{n}[f] \in \mathbb{C}^{n \times n}$ is associated with a Lebesgue integrable function $f$ via its Fourier series
\[
f(\theta)\sim \sum_{k=-\infty}^{\infty}a_{k}e^{\mathbf{i}k\theta}
\]
defined on $[-\pi,\pi]$. Thus, we have
 \[ T_{n}[f]=\begin{bmatrix}{}
a_0 & a_{-1} & \cdots & a_{-n+2} & a_{-n+1} \\
a_1 & a_0 & a_{-1}   &  & a_{-n+2} \\
\vdots & a_1 & a_0 & \ddots & \vdots \\
a_{n-2} &  & \ddots & \ddots & a_{-1} \\
a_{n-1} & a_{n-2} &\cdots & a_1 & a_0
\end{bmatrix}, \] where
\[ a_{k}=\frac{1}{2\pi} \int_{-\pi} ^{\pi}f(\theta) e^{-\mathbf{i} k \theta } \,d\theta,\quad k=0,\pm1,\pm2,\dots  \] are the Fourier coefficients of $f$. The function $f$ is called the \emph{generating function} of $T_n[f]$. If $f$ is complex-valued, then $T_n[f]$ is non-Hermitian for all $n$ sufficiently large. If $f$ is real-valued, then $T_n[f]$ is Hermitian for all $n$. If $f$ is real-valued and nonnegative, but not identically zero almost everywhere, then $T_n[f]$ is Hermitian positive definite for all $n$. If $f$ is real-valued and even, $T_n[f]$ is (real) symmetric for all $n$ \cite{MR2108963,MR2376196}.

Throughout this note, we assume that $f\in L^1([-\pi, \pi])$ and follow all standard notation and terminology introduced in \cite{MR3674485}: let $C_c(\mathbb{C})$ (or $C_c(\mathbb{R})$) be the space of complex-valued continuous functions defined on $\mathbb{C}$ (or $\mathbb{R}$) with bounded support and $\phi$ be a functional, i.e. any function defined on some vector space which takes values in $\mathbb{C}$. Also, if $g:D\subset \mathbb{R}^k \to \mathbb{K}$ ($\mathbb{R}$ or $\mathbb{C}$) is a measurable function defined on a set $D$ with $0<\mu_k(D)<\infty$, the functional $\phi_g$ is denoted such that
\[
\phi_g:C_c(\mathbb{K})\to \mathbb{C}~~\text{and}~~\phi_g(F)=\frac{1}{\mu_k(D)}\int_D F(g(\mathbf{x}))\,d\mathbf{x}.
\]

%(Singular value and eigenvalue distribution of a matrix sequence)
\begin{definition}\cite[Definition 3.1]{MR3674485}
Let $\{A_n\}_n$ be a matrix sequence.
\begin{enumerate}

\item We say that $\{A_n\}_n$ has an asymptotic singular value distribution described by a functional $\phi:C_c(\mathbb{R})\to \mathbb{C},$ and we write $\{A_n\}_n \sim_{\sigma}\phi,$ if
\[
\lim_{n\to \infty} \frac{1}{n}\sum_{j=1}^{n}F(\sigma_j(A_n))=\phi(F),~~\forall F \in C_c(\mathbb{R}).
\] If $\phi=\phi_{|f|}$ for some measurable $f:D \subset \mathbb{R}^k \to \mathbb{C}$ defined on a set $D$ with $0<\mu_k(D)<\infty,$ we say that $\{A_n\}_n$ has an asymptotic singular value distribution described by $f$ and we write $\{A_n\}_n \sim_{\sigma} f.$% In this case, the function $f$ is referred to as the singular value symbol of the matrix-sequence $\{A_n\}_n$.

\item We say that $\{A_n\}_n$ has an asymptotic eigenvalue (or spectral) distribution described by a function $\phi:C_c(\mathbb{R})\to \mathbb{C},$ and we write $\{A_n\}_n \sim_{\lambda}\phi,$ if
\[
\lim_{n\to \infty} \frac{1}{n}\sum_{j=1}^{n}F(\lambda_j(A_n))=\phi(F),~~\forall F \in C_c(\mathbb{C}).
\] If $\phi=\phi_{f}$ for some measurable $f:D \subset \mathbb{R}^k \to \mathbb{C}$ defined on a set $D$ with $0<\mu_k(D)<\infty,$ we say that $\{A_n\}_n$ has an asymptotic eigenvalue (or spectral) distribution described by  $f$ and we write $\{A_n\}_n \sim_{\lambda} f.$% In this case, the function $f$ is referred to as the eigenvalue (or spectral) symbol of the matrix-sequence $\{A_n\}_n$.
\item Let $\{A_{n}\}_{n}$ be a matrix-sequence. We say that $\{A_{n}\}_{n}$ is
\textit{sparsely vanishing (s.v.)} if for every $M>0$ there exists $n_{M}$
such that, for $n\geq n_{M}$,
\[
\frac{\#\left\{ i\in \{1,...,n\}:\sigma _{i}(A_{n})<1/M\right\} }{n}
\leq r(M)
\]
where\ $\lim_{M\rightarrow \infty }r(M)=0.$

Note that $\{A_{n}\}_{n}$ is \textit{sparsely vanishing if and only} if
\[
\lim_{M\rightarrow \infty }\lim \sup_{n\rightarrow \infty }\frac{\#\left\{
i\in \{1,...,n\}:\sigma _{i}(A_{n})<1/M\right\} }{n}=0,
\]%
i.e.
\[
\lim_{M\rightarrow \infty }\lim \sup_{n\rightarrow \infty }\frac{1}{n}
\sum_{i=1}^{n}\chi _{\lbrack 0,1/M)}\left( \sigma _{i}(A_{n})\right) =0.
\]
Finally we say that $\{A_{n}\}_{n}$ is \textit{sparsely vanishing (s.v.) in the sense of the eigenvalues} if in the previous two displayed equations the quantity $\sigma _{i}(A_{n})$ is replaced by $|\lambda _{i}(A_{n})|$ for $i=1,\ldots,n$.
\end{enumerate}
\end{definition}

The following result holds (see e.g. \cite{MR3674485}).

\begin{theorem}\label{lem:sparsely_vanishing}
The following statements are true.
\begin{enumerate}
\item
Assume $\{A_n\}_n \sim_{\sigma} f.$  Then $\{A_n\}_n$ is sparsely vanishing if and only if $f$ is sparsely vanishing.
\item
Assume $\{A_n\}_n \sim_{\lambda} f.$  Then $\{A_n\}_n$ is sparsely vanishing in the eigenvalues sense if and only if $f$ is sparsely vanishing.
\item Assume $\{A_n\}_n$ is given and assume that every matrix $A_n$ is normal. Then $\{A_n\}_n$ is sparsely vanishing if and only if  $\{A_n\}_n$ is sparsely vanishing in the eigenvalues sense.
\end{enumerate}
\end{theorem}

The generalized Szeg{\H{o}} theorem that describes the singular value and spectral distribution of Toeplitz sequences is given as follows:

\begin{theorem}\cite[Theorem 6.5]{MR3674485}\label{lem:main}
Suppose $f \in L^{1}([-\pi,\pi])$. Let $T_n[f]$ be the Toeplitz matrix generated by $f$. Then $$\{T_n[f]\}_n \sim_{\sigma} f.$$ If moreover $f$ is real-valued, then
$$\{T_n[f]\}_n \sim_{\lambda} f.$$
\end{theorem}

We introduce the following definitions and lemma in order to prove our claim on $\{Y_nT_n[f]\}_n$.
%(Approximating class of sequences)
\begin{definition}\cite[Definition 5.1]{MR3674485}\label{def:ACS}
Let $\{A_n\}_n$ be a matrix sequence and let $\{\{B_{n,m}\}_n\}_m$ be a sequence of matrix sequences. We say that $\{\{B_{n,m}\}_n\}_m$ is an \textit{approximating class of sequences (a.c.s)} for $\{A_n\}_n$ if the following condition is met: for every $m$ there exists $n_m$ such that, for $n \geq n_m$,
\[
A_n=B_{n,m}+R_{n,m}+N_{n,m},
\]
\[\text{rank}~R_{n,m}\leq c(m)n~~\text{and}~~\|N_{n,m}\|\leq\omega(m),
\]
where $n_m$, $c(m)$ and $\omega(m)$ depend only on $m$, and \[\lim_{m\to\infty}c(m)=\lim_{m\to\infty}\omega(m)=0.\]
\end{definition}

We use $\{B_{n,m}\}_n\xrightarrow{\text{a.c.s.\ wrt\ $m$}}\{A_n\}_n$ to denote that $\{\{B_{n,m}\}_n\}_m$ is an a.c.s for $\{A_n\}_n$.

\begin{definition}
Let $f_m,f:D \subset \mathbb{R}^k \to \mathbb{C}$ be measurable functions. We say that $f_m \to f$ in measure if, for every $\epsilon > 0$,
\[
\lim_{m \to \infty} \mu_k\{|f_m-f|>\epsilon\}=0.
\]
\end{definition}

\begin{lemma}\cite[Corollary 5.1]{MR3674485}\label{lem:Corollary5.1}
Let $\{A_n\}_n, \{B_{n,m}\}_n$ be matrix sequences and let $f,f_m:D \subset \mathbb{R}^k \to \mathbb{C}$ be measurable functions defined on a set $D$ with $0<\mu_k(D)<\infty$. Suppose that

\begin{enumerate}
\item $\{B_{n,m}\}_n \sim_{\sigma}  f_m$ for every $m$,
\item $\{B_{n,m}\}_n\xrightarrow{\text{a.c.s.\ wrt\ $m$}}\{A_n\}_n$,
\item $f_m \to f$ in measure.
\end{enumerate}

Then $\{A_{n}\}_n \sim_{\sigma}  f$.
\\

Moreover, if the first assumption is replaced by $\{B_{n,m}\}_n \sim_{\lambda}  f_m$ for every $m$, the other two assumptions are left unchanged, and all the involved matrices are Hermitian, then  $\{A_{n}\}_n \sim_{\lambda}  f$.
\end{lemma}

\section{Preliminaries on matrix analysis}

The following definitions and lemma will be used in the proof of our main result.

\begin{definition}\cite{doi:10.1137/140974213}
For any circulant matrix $C_n \in \mathbb{C}^{n\times n}$, the \emph{absolute value circulant matrix} $|C_n|$ of $C_n$ is defined by
\begin{eqnarray}\nonumber
|C_n|&=& (C_n^* C_n)^{1/2}\\\nonumber
&=&(C_n C_n^*)^{1/2}\\\nonumber
&=&F_n|\Lambda_n|F_n^*,
\end{eqnarray} where 
$F_n = \left(\frac{\omega^{jk}}{\sqrt{n}} \right)_{j,k=0}^{n-1},~ \omega=e^{-\mathbf{i}{2\pi\over n}}$, and $|\Lambda_n|$ is the diagonal matrix in the eigendecomposition of $C_n$ with all entries replaced by their magnitude.
\end{definition}

\begin{remark}
By definition, $|C_n|$ is Hermitian positive definite provided that $C_n$ is nonsingular.
\end{remark}

\begin{definition}\cite[Definition 1.3.16]{MR2978290}
A \textit{family} $\mathfrak{F} \subseteq \mathbb{C}^{n \times n}$ of matrices is an arbitrary (finite or infinite) set of matrices, and a \textit{commuting family} is one in which each pair in the set commutes under multiplication.
\end{definition}

\begin{theorem}\cite[Theorem 4.1.6]{MR2978290}\label{lem:cmsd}
Let $\mathfrak{F}$ be a given family of Hermitian matrices. There exists a unitary matrix $U_n$ such that $U_n^*A_nU_n$ is diagonal for all $A_n \in \mathfrak{F}$ if and only if $A_n B_n=B_n A_n$ for all $A_n, B_n \in \mathfrak{F}$.
\end{theorem}

\section{Main results}

We are now ready to provide our main result on the singular and spectral distribution of $\{Y_nT_n[f]\}_n$.

In Theorem \ref{thm:main_theorem}, we furnish the singular value distribution of $\{Y_nT_n[f]\}_n$ and the asymptotic inertia of $Y_nT_n[f]$ that is an evaluation of the number of positive, negative, and zero eigenvalues.

The section is concluded by a few comments and remarks on the impact of the result.

\begin{theorem}\label{thm:main_theorem}
Suppose $f \in L^1([-\pi,\pi])$ with real Fourier coefficients and $Y_n \in \mathbb{R}^{n \times n}$ is the anti-identity matrix. Let $T_n[f]\in \mathbb{R}^{n \times n}$ be the Toeplitz matrix generated by $f$. Then $$\{Y_nT_n[f]\}_n \sim_{\sigma}  f.$$ Moreover, $Y_nT_n[f]$ is (real) symmetric and if $f$ is sparsely vanishing then
\[
|n^{+}(Y_nT_n[f])-n^{-}({Y_nT_n[f]})|=o(n),
\]
with $n^{+}(\cdot)$ and $n^{-}(\cdot)$ denoting the number of positive and the negative eigenvalues of its argument, respectively.
\end{theorem}

\begin{proof}
Showing the first part of the statement that $\{Y_nT_n[f]\}_n \sim_{\sigma}  f$ is in fact trivial. We recall the fact that premultiplying a matrix by a unitary matrix does not change its singular values. As $Y_n$ is a permutation matrix and we know from Theorem \ref{lem:main} or \cite{MR1481397,MR1671591} that $\{T_n[f]\}_n \sim_{\sigma}  f$ holds, we can readily conclude that $\{Y_nT_n[f]\}_n \sim_{\sigma}  f$. However, for the completeness of the theory of GLT sequences, we provide a detailed proof of the claim in the following. In contrast to the first part, showing the second part of the statement concerning the number of negative/positive eigenvalues of $Y_nT_n[f]$ is not as straightforward.\\
\\
{\bf Proof of the first part}\\
We first show that the theorem holds for trigonometric polynomials, and then show that the statement also holds for $f\in L^1([-\pi,\pi])$ based on an approximation result.

Given a trigonometric polynomial $p_M$ where $M$ is a nonnegative integer, we have
\[
 p_M(\theta) = \sum_{k=-M}^{k=M}\rho_k e^{\mathbf{i}k\theta}.
\]We consider Strang's circulant matrix for $p_M$
\[
C_n[p_M]=\sum_{k=-M}^{M}\rho_k \Pi_n^{k}
\]where
\[
\Pi_n=\begin{bmatrix}{}
0&&&1\\
1&\ddots&&\\
&\ddots&\ddots&\\
&&1&0
\end{bmatrix} \in \mathbb{R}^{n \times n}
\] is the basic circulant matrix. As shown in the proof of Theorem 6.5 in \cite{MR3674485}, $\{C_n[p_M]\}_n \sim_{\sigma} p_M$ holds and this fact will be used in the later step of this proof.

We are now to prove our claim on $\{Y_nT_n[p_M]\}_n$. Recalling the definition of absolute circulant matrices, we have $|C_n[p_M]|=F_n|\Lambda_n[p_M]|F_n^*$. Thus
\begin{eqnarray}\nonumber
Y_nC_n[p_M]&=&Y_nF_n\Lambda_n[p_M]F_n^*\\\nonumber
&=&Y_n\underbrace{F_n\widetilde{\Lambda}_n[p_M] F_n^*}_{\widetilde{C}_n[p_M]} F_n|\Lambda_n[p_M]|F_n^*\\\nonumber
&=&\underbrace{Y_n\widetilde{C}_n[p_M]}_{Q_n[p_M]} |C_n[p_M]|\\\nonumber
&=&Q_n[p_M]|C_n[p_M]|,
\end{eqnarray}
where $\widetilde{\Lambda}_n[p_M]$ is the diagonal matrix having the signs of eigenvalues of ${\Lambda}_n[p_M]$ as its eigenvalues.

Note that all circulant matrices are normal, i.e. they commute and $Y_n|C_n|$ is symmetric for a real $|C_n|$, i.e. $Y_n|C_n|=(Y_n|C_n|)^T=|C_n|^T Y_n^T=|C_n|Y_n.$ Therefore,
\begin{eqnarray}\nonumber
Q_n[p_M]|C_n[p_M]|&=&Y_n\widetilde{C}_n[p_M]|C_n[p_M]|\\\nonumber
&=&Y_n|C_n[p_M]|\widetilde{C}_n[p_M]\\\nonumber
&=&|C_n[p_M]|Y_n\widetilde{C}_n[p_M]\\\nonumber
&=&|C_n[p_M]|Q_n[p_M],
\end{eqnarray}
namely $Q_n[p_M]$ and $|C_n[p_M]|$ commute. Moreover, both $Q_n[p_M]$ and $|C_n[p_M]|$ are symmetric: as $\widetilde{C}_n[p_M]$ is a real circulant matrix, $Q_n[p_M]=Y_n \widetilde{C}_n[p_M]$ is symmetric; $|C_n[p_M]|$ is symmetric by definition. Therefore, by Theorem \ref{lem:cmsd}, there exists a unitary matrix $U_n$ such that both $U_n^{*}Q_n[p_M]U_n$ and $U_n^{*}|C_n[p_M]|U_n$ are diagonal.

Furthermore, $Q_n[p_M]$ is orthogonal as
\begin{eqnarray}\nonumber
Q_n[p_M]^T Q_n[p_M]&=&(Y_n \widetilde{C}_n[p_M])^T(Y_n \widetilde{C}_n[p_M])\\\nonumber
&=&\widetilde{C}_n[p_M]^TY_n^TY_n \widetilde{C}_n[p_M]\\\nonumber
&=&\widetilde{C}_n[p_M]^T \widetilde{C}_n[p_M]\\\nonumber
&=&(F_n\widetilde{\Lambda}_n[p_M] F_n^*)^*(F_n\widetilde{\Lambda}_n[p_M] F_n^*)\\\nonumber
&=&F_n(\widetilde{\Lambda}_n[p_M])^* F_n^*F_n\widetilde{\Lambda}_n[p_M] F_n^*\\\nonumber
&=&F_n|\widetilde{\Lambda}_n[p_M]|^2 F_n^*\\\nonumber
&=&F_n F_n^*\\\nonumber
&=&I_n.
\end{eqnarray}

With $Q_n[p_M]$ being both symmetric and orthogonal, we have $Q_n[p_M]=U_n\Sigma[p_M] U_n^{*}$ where $\Sigma_n[p_M]$, having only eigenvalues $\pm 1$, is the diagonal matrix in the eigendecomposition of $Q_n[p_M]$.

As a result,
\begin{eqnarray}\nonumber
Y_nC_n[p_M]&=&Q_n[p_M]|C_n[p_M]|\\\nonumber
&=&U_n\Sigma_n[p_M] U_n^{*} U_n|\Lambda_n[p_M]| U_n^{*}\\
&=&U_n\underbrace{\Sigma_n[p_M]|\Lambda_n[p_M]|}_{\Upsilon_n[p_M]} U_n^{*}\label{eqn:spectral_YC-1}
\end{eqnarray}
where $\Upsilon_n[p_M]$ is the diagonal matrix in the eigendecomposition of $Y_nC_n[p_M]$, having only eigenvalues $\pm |\lambda_j|$ with $\lambda_j$ being the $j$-th eigenvalue of $C_n[p_M]$. The eigenvalues of $Y_nC_n[p_M]$ are therefore all of the form $\pm |\lambda_j|$ whilst its singular values are of $|\lambda_j| $.

Recalling $\{C_n[p_M]\}_n \sim_{\sigma}  p_M$, we consequently have $\{|C_n[p_M]|\}_n \sim_{\sigma}  p_M$ as the singular values of $C_n[p_M]$ are the same as those of $|C_n[p_M]|$. From (\ref{eqn:spectral_YC-1}), the singular values of $Y_nC_n[p_M]$ are the same as those of $|C_n[p_M]|$. Thus, we have
\begin{equation}\label{enq:YCpolysing}
\{Y_nC_n[p_M]\}_n \sim_{\sigma} p_M.
\end{equation}

Now, we are going to prove
$\{Y_nC_n[p_M]\}_n\xrightarrow{\text{a.c.s.\ wrt\ $m$}}\{Y_nT_n[p_M]\}.$

As shown in the proof of Theorem 6.5 of \cite{MR3674485}, we readily have 
\[
\{C_n[p_M]\}_n\xrightarrow{\text{a.c.s.\ wrt\ $m$}}\{T_n[p_M]\},
\]
 namely for every $m$ there exists $n_m$ such that for $n \geq n_m$
\begin{equation}\label{eqn:TCRN}
T_n[p_M]=C_{n,m}[p_M]+R_{n,m}[p_M]+N_{n,m}[p_M],
\end{equation}
\[
\text{rank}~R_{n,m}[p_M]\leq c(m)n\ \ \text{and}\ \ \|N_{n,m}[p_M]\|\leq\omega(m),
\]
where $n_m$, $c(m)$ and $\omega(m)$ depend only on $m$, and
\[
\lim_{m\to\infty}c(m)=\lim_{m\to\infty}\omega(m)=0.
\]

Multiplying the both sides of (\ref{eqn:TCRN}) by $Y_n$, we have for every $m$ there exists $n_m$ such that for $n \geq n_m$
\[
Y_nT_n[p_M]=Y_nC_{n,m}[p_M]+Y_nR_{n,m}[p_M]+Y_nN_{n,m}[p_M],
\]
\[
\text{rank}(Y_nR_{n,m}[p_M])=\text{rank}~R_{n,m}[p_M]\leq c(m)n,
\]
and
\[
\|Y_nN_{n,m}[p_M]\|=\|N_{n,m}[p_M]\|\leq\omega(m).
\] Therefore, by Definition \ref{def:ACS}, we have

\begin{equation}\label{eqn:YCACS}
\{Y_nC_n[p_M]\}_n\xrightarrow{\text{a.c.s.\ wrt\ $m$}}\{Y_nT_n[p_M]\}.
\end{equation}

Hence, combining (\ref{enq:YCpolysing}) and (\ref{eqn:YCACS}), by Lemma \ref{lem:Corollary5.1} we have
\begin{equation}\label{enq:YTpolysing}
\{Y_nT_n[p_M]\}_n \sim_{\sigma}  p_M.
\end{equation}

As the set of trigonometric polynomials is dense in $L^1([-\pi,\pi])$, there exists a sequence of trigonometric polynomials $\{p_M\}_M$ such that $p_M \to f$ in $L^1([-\pi,\pi])$.

As shown in the proof of Theorem 6.5 \cite{MR3674485},
$\{T_n[p_M]\}_n\xrightarrow{\text{a.c.s.\ wrt\ $M$}}\{T_n[f]\}$ holds. Equivalently, we have
\begin{equation}\label{enq:YTacs}
\{Y_nT_n[p_M]\}_n\xrightarrow{\text{a.c.s.\ wrt\ $M$}}\{Y_nT_n[f]\}.
\end{equation}

With (\ref{enq:YTpolysing}) and (\ref{enq:YTacs}), we conclude by Lemma \ref{lem:Corollary5.1} that $$\{Y_nT_n[f]\}_n \sim_{\sigma}  f.$$
\\
{\bf Proof of the second part}\\
We let $H_\nu[f,-]$ be the $\nu\times \nu$ Hankel matrix generated by $f$ containing the negative Fourier coefficients from $a_{-1}$ in position $(1,1)$ to $a_{-2\nu+1}$ in position $(\nu,\nu)$. Analogously, we let $H_\nu[f,+]$ be the $\nu\times \nu$ Hankel matrix generated by $f$ containing the positive Fourier coefficients from $a_{1}$ in position $(1,1)$ to $a_{2\nu-1}$ in position $(\nu,\nu)$.

We start by considering the case of even $n$ and writing $Y_nT_n[f]$ as a $2\times 2$ block matrix of size $n=2\nu$, i.e.
\[
Y_nT_n[f] = \left[\begin{array}{cc}
      Y_\nu H_\nu[f,+] Y_\nu   & Y_\nu T_\nu[f] \\
      Y_\nu T_\nu[f]  & H_\nu[f,-] \end{array}\right].
\]
It is worth noticing that for Lebesgue integrable $f$, $H_\nu[f,+]$ is exactly the Hankel matrix generated by $f$ according
to the definition given in \cite{FasinoTilli}: in that paper it was proven that $\{H_\nu[f,+]\}_n\sim_{\sigma}  0$. Since in our setting  $H_\nu[f,+]$  is symmetric for every $\nu$, it follows that  $\{H_\nu[f,+]\}_n\sim_{\lambda}  0$. Hence, with $Y_\nu$ being both symmetric and orthogonal, we have
\[
\{Y_\nu H_\nu[f,+] Y_\nu\}_n\sim_{\lambda,\sigma}  0.
\]
Similarly, we have
\[
 \{ H_\nu[f,-] \}_n\sim_{\lambda,\sigma}  0
\]
as $H_\nu[f,-]=H_\nu[\bar f,+]$ and $\bar f$ (being the conjugate of $f$) is Lebesgue integrable if and only if $f$ is Lebesgue integrable.

Therefore, the matrix sequence $\{Y_nT_n[f]\}_n$ can be written as the sum of a matrix sequence whose eigenvalues are clustered at zero that is
\[
\{ R_n \}_n =\left\{
\left[\begin{array}{cc}
      Y_\nu H_\nu[f,+] Y_\nu   &  O\\
      O  & H_\nu[f,-] \end{array}\right]\right\}_n
\]
and a matrix sequence
\[
\left\{
\left[\begin{array}{cc}
      O  &  Y_\nu T_\nu[f] \\
      Y_\nu T_\nu[f]   & O \end{array}\right]
\right\}_n
\]
whose eigenvalues are $\pm \sigma_j(Y_\nu T_\nu[f])=\pm \sigma_j( T_\nu[f])$,
$j=1,\ldots,\nu$. In fact, if we consider the singular value decomposition $Y_\nu T_\nu[f] = U \Sigma V^*=V \Sigma U^*$, where $U^*U=V^*V=I$ and the diagonal matrix $\Sigma$ with the singular values of $Y_\nu T_\nu[f]$, then
\[
\left[\begin{array}{cc}
      O  &  Y_\nu T_\nu[f] \\
      Y_\nu T_\nu[f]   & O \end{array}\right]
= \left[\begin{array}{cc}
      U &  O \\
      O & V \end{array}\right]
     \left[\begin{array}{cc}
      O  &  \Sigma \\
     \Sigma   & O \end{array}\right]
    \left[\begin{array}{cc}
      U^* &  O \\
      O & V^* \end{array}\right]
\]
which is similar to
\[
\left[\begin{array}{cc}
      \Sigma &  O \\
      O & -\Sigma \end{array}\right].
\]
Owing to the relation $\{Y_nT_n[f]\}_n \sim_{\sigma}  f$ and the assumption that $f$ is sparsely vanishing, the desired result follows by using the classical Cauchy interlacing theorem, Theorem \ref{lem:sparsely_vanishing}, and the relation $\{ R_n \}_n\sim_{\lambda}  0$.

In the case where $n$ is odd, the analysis is of the same type as before with a few slight modifications.

		By setting $\nu=\lfloor n/2 \rfloor$ and $\mu= \lceil n/2 \rceil$, we have
		\[
		Y_nT_n[f] = \left[\begin{array}{ccc}
		Y_\nu H_\nu[f \cdot e^{-\mathbf{i}\theta},+] Y_\nu & v   & Y_\nu T_\nu[f] \\
		v^T & a_0 & w^T \\
		Y_\nu T_\nu[f]  & w & H_\nu[f \cdot e^{\mathbf{i}\theta},-] \end{array}\right],
		\]
		provided that $n\neq 1$.
		Therefore, the matrix sequence $\{Y_nT_n[f]\}_n$ can be written as the sum of the matrix sequence whose eigenvalues are clustered at zero, that is $\{ E_n \}_n$, where $E_n=E_n'+E_n''$ with
		\[
		E_n' =
		\left[\begin{array}{cc}
		Y_\mu H_\mu[f\cdot e^{\mathbf{i}  \theta },+] Y_\mu   &  O\\
		O  & H_\nu[f \cdot e^{\mathbf{i}\theta},-] \end{array}\right],
		\]
		\[
		Y_\mu H_\mu[f\cdot e^{\mathbf{i}  \theta },+] Y_\mu =
		\left[\begin{array}{cc}
		Y_\nu H_\nu[f \cdot e^{-\mathbf{i}\theta},+] Y_\nu & v   \\
		v^T & a_0
		\end{array}\right],
		\]
		\[
		E_n'' =
		\left[\begin{array}{ccc}
		O & {\bf 0}   & O \\
		{\bf 0}^T & 0 & w^T \\
		O  & w & O \end{array}\right],
		\]
		and the matrix sequence
		\[
		\left\{
		\left[\begin{array}{ccc}
		O  &  {\bf 0} & Y_\nu T_\nu[f] \\
		{\bf 0}^T & 0 &  {\bf 0}^T \\
		Y_\nu T_\nu[f] & {\bf 0}   & O \end{array}\right]
		\right\}_n
		\]
		whose eigenvalues are $0$ with multiplicity $1$ and $\pm \sigma_j(Y_\nu T_\nu[f])$,
		$j=1,\ldots,\nu$. Note that we have $\sigma_j(Y_\nu T_\nu[f])= \sigma_j( T_\nu[f])$,
		$j=1,\ldots,\nu$, again from the singular value decomposition of $Y_\nu T_\nu[f]$. Owing to the relation $\{Y_nT_n[f]\}_n \sim_{\sigma}  f$ and the assumption that $f$ is sparsely vanishing, the desired result follows by using the classical Cauchy interlacing theorem, Theorem \ref{lem:sparsely_vanishing}, and the relation $\{ R_n \}_n\sim_{\lambda}  0$. \qed
\end{proof}

In other words, the singular value distribution of $\{Y_nT_n[f]\}_n$ is described by $|f|$.  Also, as $Y_nT_n[f]$ is symmetric, its eigenvalues are essentially distributed as $\pm |f|$ asymptotically with roughly half of them being negative/positive. Consequently, provided that $f$ is sparsely vanishing, $Y_nT_n[f]$ is indefinite for sufficiently large $n$ by Theorem \ref{thm:main_theorem}. 

We observe that in the case where $f$ is a trigonometric polynomial with real coefficients and not identically zero, we can say a bit more. In that case, by the fundamental theorem of algebra, the function $f$ is guaranteed to be sparsely vanishing. In addition, every matrix of the perturbing sequence $\{ R_n \}_n$ has a rank bounded by a constant independent of $n$ and linearly depending on the degree of $f$. As a conclusion, by following the same steps of our main theorem, we deduce
\begin{equation}\label{improved}
|n^{+}(Y_nT_n[f])-n^{-}({Y_nT_n[f]})|=O(1).
\end{equation}

\section{Numerical examples}

To illustrate the second part of Theorem \ref{thm:main_theorem}, we provide the following numerical examples, where for the computation we use the Matlab \textbf{eig} function. Figure \ref{fig:Bi_YT_n} shows the spectral distribution of $Y_nT_n[f]$, where $$T_n[f]=\begin{bmatrix}{}
2 & & &   \\
1 & 2 &  &    \\
& \ddots& \ddots &  \\
&   & 1 & 2
\end{bmatrix}\in \mathbb{R}^{n \times n}$$ generated by $f(\theta)=2+e^{\mathbf{i}\theta}$. Figure \ref{fig:Gracr_YT_n} shows the spectral distribution of $Y_nT_n[f]$, where \[
T_n[f]=\begin{bmatrix}{}
1 & 1 & 1 & 1&  &  \\
-1 & \ddots & \ddots &\ddots & \ddots  &  \\
& \ddots &\ddots & \ddots & \ddots & 1 \\
& &\ddots & \ddots & \ddots & 1 \\
& & & \ddots & \ddots & 1 \\
& & & & -1 & 1
\end{bmatrix}\in \mathbb{R}^{n \times n}
\] being a Grcar matrix. At last, Figure \ref{fig:Thrid_YT_n} shows the spectral distribution of $Y_nT_n[f]$, where \[
T_n[f]=\begin{bmatrix}{}
-4 & 6 & -4 & 1&  &  \\
1 & \ddots & \ddots &\ddots & \ddots  &  \\
& \ddots &\ddots & \ddots & \ddots & 1 \\
& &\ddots & \ddots & \ddots & -4 \\
& & & \ddots & \ddots & 6 \\
& & & & 1 & -4
\end{bmatrix}\in \mathbb{R}^{n \times n}
\] generated by $f(\theta)=e^{-3\mathbf{i}\theta}-4e^{-2\mathbf{i}\theta}+6e^{-\mathbf{i}\theta}-4+e^{\mathbf{i}\theta}$. From all figures, we observe that the spectral distribution of $Y_nT_n[f]$ behaves as described by our theorem. Moreover, considering the same Toeplitz matrices, Table \ref{tab:YT_n_eigenvalues} verifies our predicted number of negative/positive eigenvalues of $Y_nT_n[f]$ at different $n$.

From the previous examples, we might argue that for even $n$ the number of negative eigenvalues is exactly equal to that of positive eigenvalues so that Theorem  \ref{thm:main_theorem} can be made more precise. In fact, this is not always the case and indeed our theorem is sharp as illustrated below.

We consider the following matrix
	$$ 
	T_n[1+6 \cos{\theta}] = \left[
	\begin{array}{cccc}
		1&3&&\\
		3&\ddots&\ddots&\\
		&\ddots&\ddots&3\\
		&&3&1 
	\end{array} \right] \in \mathbb{R}^{n \times n}.
	$$
Premultiplying it by $Y_n$ gives
	$$ 
	Y_n T_n[1+6 \cos{\theta}] = \left[
	\begin{array}{cccc}
		&&3&1\\
		&\iddots&\iddots&3\\
		3&\iddots&\iddots&\\
		1&3&& 
\end{array} \right]\in \mathbb{R}^{n \times n}
	$$
and the real symmetric matrix $R_n$ considered in Theorem \ref{thm:main_theorem} has rank $2$ with $2$ positive eigenvalues and $n-2$ zero eigenvalues. The reasoning is supported by the numeric: the eigenvalues for $n=6$  are
$$
\Lambda(Y_6 T_6[1+6 \cos{\theta}])= \left( \begin{array}{c}
-4.740938811152401\\
-2.740938811152402\\
0.335125603737888\\
2.335125603737888\\
4.405813207414513\\
6.405813207414515
\end{array} \right)
$$ 
and hence $n^{+}(Y_6 T_6[1+6 \cos{\theta}])-n^{-}(Y_6 T_6[1+6 \cos{\theta}]) = 2$. We emphasize that such relation is confirmed also for $n=100$ and $n=200$ so that
$$
n^{+}(Y_nT_n[1+6 \cos{\theta}])-n^{-}(Y_n T_n[1+6 \cos{\theta}]) = 2
$$
is in perfect accordance with relation (\ref{improved}). The fact that $R_n$ is positive semidefinite is the key ingredient in forcing the number of positive eigenvalues to be larger than that of negative eigenvalues and hence is a recipe for constructing specific examples of this kind.

\begin{figure}[!htb]
\centering
\subfloat[]{\includegraphics[scale=0.4]{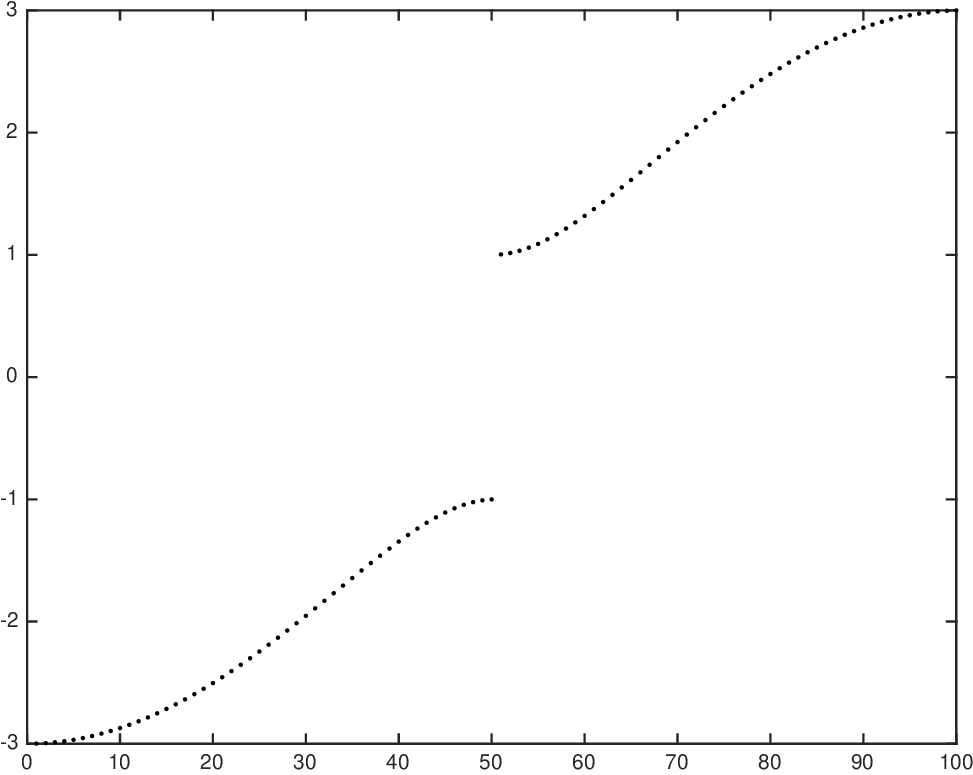}}~
\subfloat[]{\includegraphics[scale=0.4]{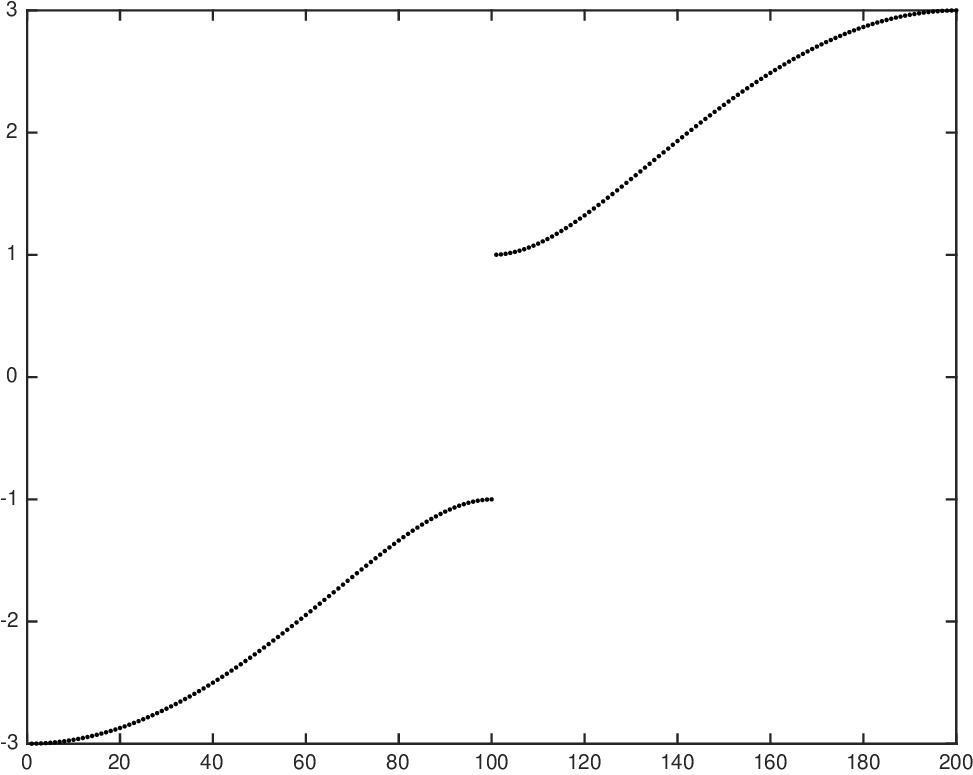}}
\caption{\footnotesize Spectral distribution of $Y_nT_n[f]$ with $T_n[f]$ generated by $f(\theta)=2+e^{\mathbf{i}\theta}$ at (a) $n=100$ and (b) $n=200$.} 
\label{fig:Bi_YT_n}
\end{figure}

\begin{figure}[!htb]
\centering
\subfloat[]{\includegraphics[scale=0.4]{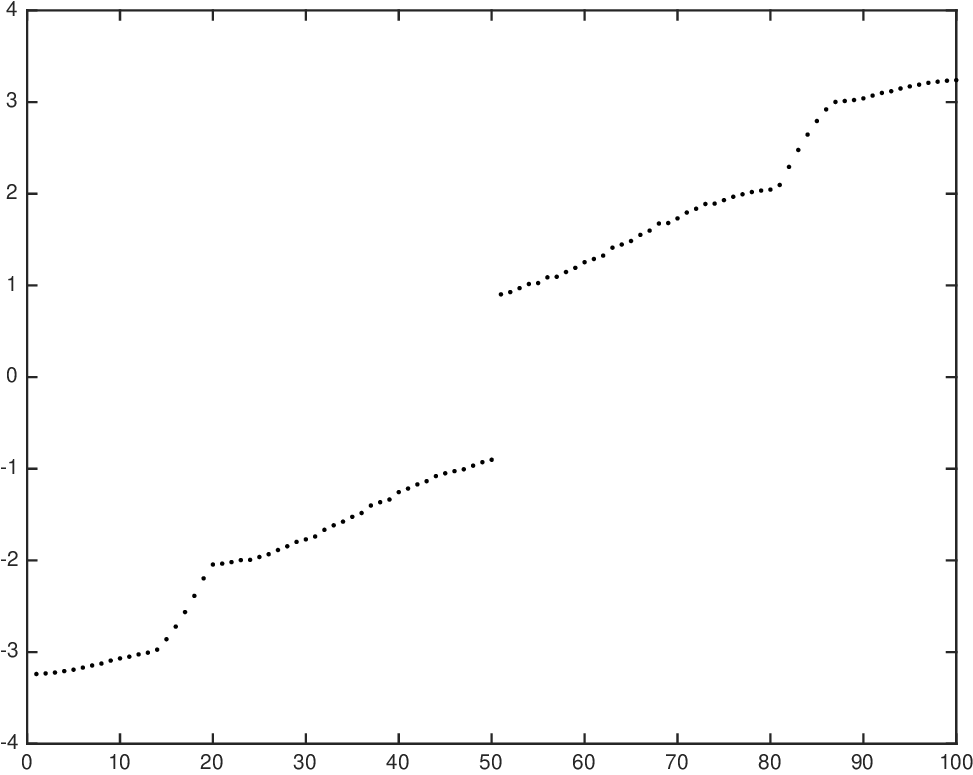}}~
\subfloat[]{\includegraphics[scale=0.4]{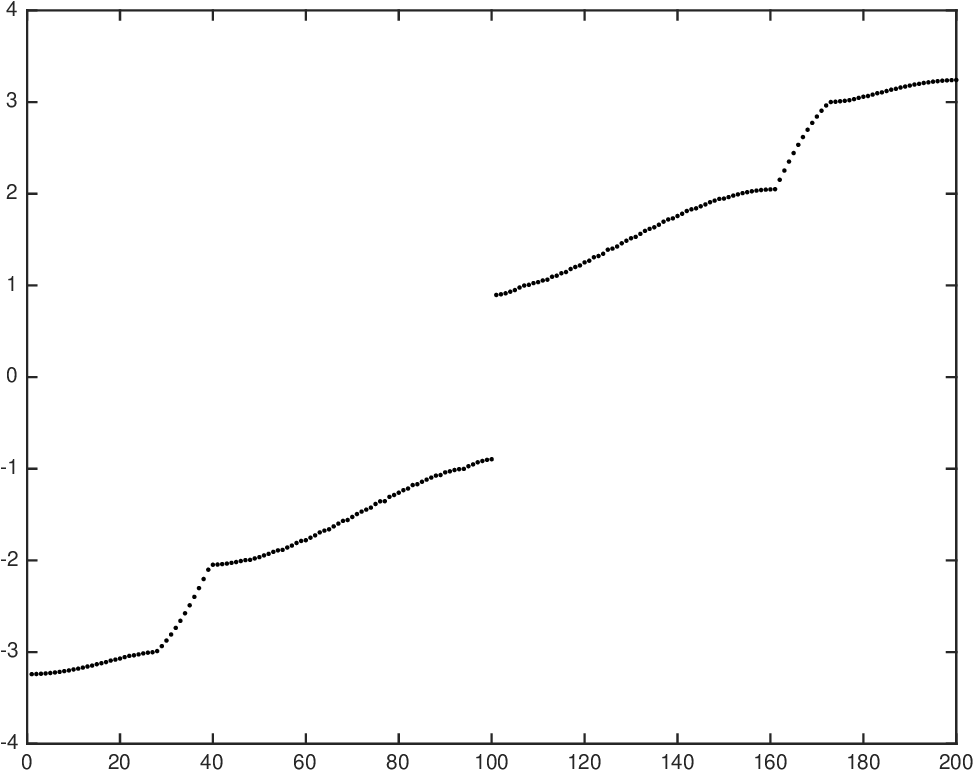}}
\caption{\footnotesize Spectral distribution of $Y_nT_n[f]$ with $T_n[f]$ being a Grcar matrix at (a) $n=100$ or (b) $n=200$.} 
\label{fig:Gracr_YT_n}
\end{figure}

\begin{figure}[!htb]
\centering
\subfloat[]{\includegraphics[scale=0.4]{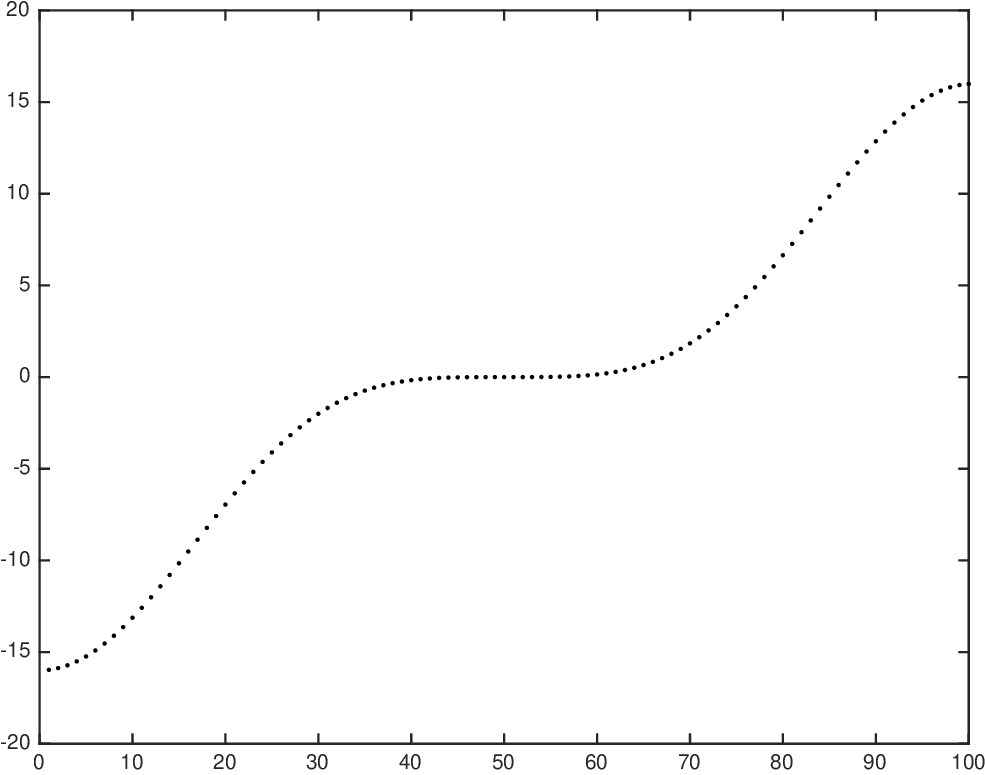}}~
\subfloat[]{\includegraphics[scale=0.4]{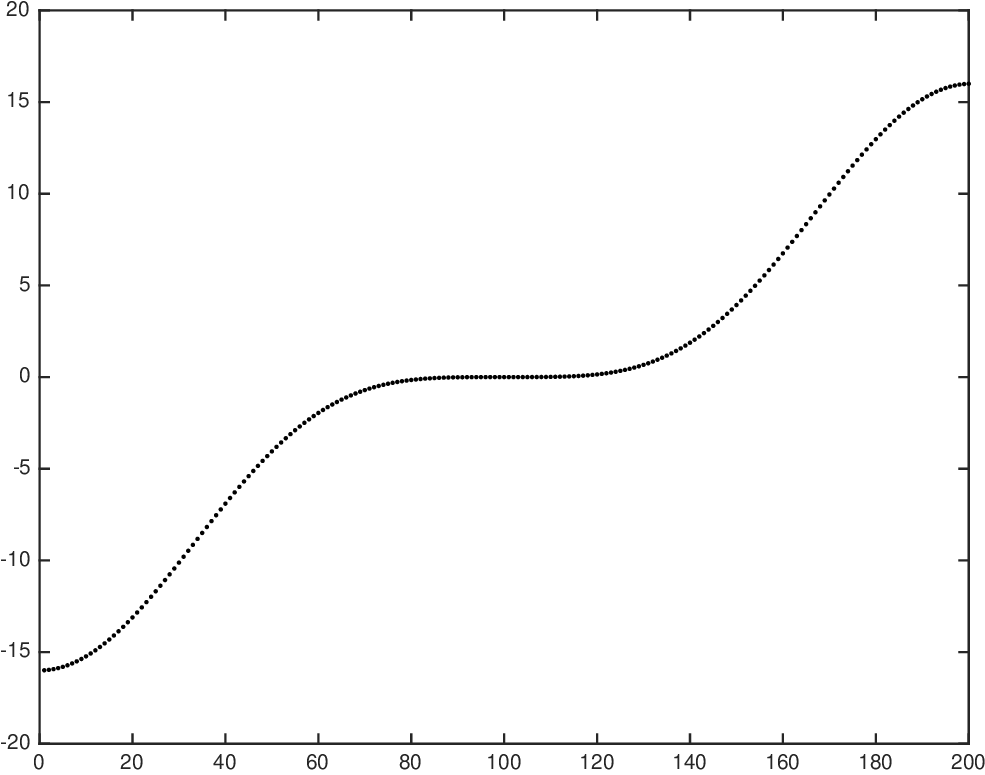}}
\caption{\footnotesize Spectral distribution of $Y_nT_n[f]$ with $T_n[f]$ generated by $f(\theta)=e^{-3\mathbf{i}\theta}-4e^{-2\mathbf{i}\theta}+6e^{-\mathbf{i}\theta}-4+e^{\mathbf{i}\theta}$ at (a) $n=100$ or (b) $n=200$.} 
\label{fig:Thrid_YT_n}
\end{figure}

\begin{table}[!htb]
\caption{Numbers of negative/positive eigenvalues of $Y_nT_n[f]$ with (a) $T_n[f]$ generated by $f(\theta)=2+e^{\mathbf{i}\theta}$, (b) $T_n[f]$ being a Grcar matrix, or (c) $T_n[f]$ generated by $f(\theta)=e^{-3\mathbf{i}\theta}-4e^{-2\mathbf{i}\theta}+6e^{-\mathbf{i}\theta}-4+e^{\mathbf{i}\theta}$.}
\label{tab:YT_n_eigenvalues}
\begin{center}
{
(a)~\begin{tabular}{|l|c|c|}\hline
 $n $ & No. of negative eigenvalues  & No. of positive eigenvalues
\\
%\hline
\hline
100 & 50 & 50 \\
200 & 100 & 100 \\
\hline
\end{tabular}\\
\vspace{0.5cm}
(b)~\begin{tabular}{|l|c|c|}\hline
 $n $ & No. of negative eigenvalues  & No. of positive eigenvalues
\\
%\hline
\hline
100 & 50 & 50 \\
200 & 100 & 100 \\
\hline
\end{tabular}\\
\vspace{0.5cm}
(c)~\begin{tabular}{|l|c|c|}\hline
 $n $ & No. of negative eigenvalues  & No. of positive eigenvalues
\\
%\hline
\hline
100 & 50 & 50 \\
200 & 100 & 100 \\
\hline
\end{tabular}
}
\end{center}
\end{table}

\section{Conclusions}

We have provided a theorem that describes the singular and spectral distribution of $\{Y_nT_n[f]\}_n$. Our result shows that the symmetrized matrix $Y_nT_n[f]$, which is in fact a class of Hankel matrices, is always indefinite for sufficiently large $n$. More precisely, asymptotically speaking, for sparsely vanishing $f$ there are always roughly half of the eigenvalues of $Y_nT_n[f]$ that are negative/positive.

As already considered in \cite{doi:10.1137/140974213}, such a sign structure of the resulting matrices suggests that a good Krylov subspace method is the preconditioned MINRES method.

\section*{References}

\bibliographystyle{plain}
%\bibliography{HonMS18}

\end{document}